\documentclass[12pt,reqno]{article}
\usepackage{amsmath, amsthm, graphicx,amsfonts, amssymb,color}
\usepackage{bm}
\usepackage{hyperref}
\newtheorem{thm}{Theorem}[section]
\newtheorem{cor}[thm]{Corollary}
\newtheorem{lem}[thm]{Lemma}

\theoremstyle{definition}

\theoremstyle{remark}
\newtheorem{rem}[thm]{Remark}
\theoremstyle{example}
\newtheorem{exm}[thm]{Example}
\numberwithin{equation}{section}
\usepackage{mathrsfs}
\newcommand{\scr}[1]{\mathscr #1}

 \def\d{\mathrm{d}}

\def\e{\scr E}

\def\cM{\mathcal M}
\def\cN{\mathcal N}

\def\bE{\mathbb E}

\def\bR{\mathbb R}

\def\d{\rm d}
\def\bg{\begin}
\def\be{\bg{equation}}
\def\de{\end{equation}}

\def\edar{\end{eqnarray}}

\def\lb{\label}

\def\l{\left}\def\r{\right}
\def\fr{\frac}
\def\alp{\alpha}
\def\bt{\beta}

\def\lmd{\lambda}

\def\rar{\rightarrow}

\def\lan{\langle}\def\ran{\rangle}

\def\[{\l[} \def\]{\r]}
\def\({\l(} \def\){\r)}

\def\bar{\overline}

 \def\beqlb{\begin{eqnarray}}\def\eeqlb{\end{eqnarray}}
 \def\beqnn{\begin{eqnarray*}}\def\eeqnn{\end{eqnarray*}}
\def\d{{\mbox{\rm d}}}

\title{\bf  {Variational principles for asymptotic variance of general Markov processes}}

\author{Lu-Jing Huang\qquad Yong-Hua Mao \qquad Tao Wang\thanks{Corresponding author:  wang\_tao@mail.bnu.edu.cn}}

\date{}

\begin{document}

 \maketitle


\bg{abstract}

A variational formula for the asymptotic variance of general Markov processes is obtained.
As application, we get a upper bound of the mean exit time of reversible Markov processes, and some comparison theorems between the reversible and non-reversible diffusion processes.

\end{abstract}




\section{Introduction and main results}

Asymptotic variance is a popular criterion to evaluate the performance of Markov processes, and widely used in Markov chain Monte Carlo(see e.g. \cite{AL19,CCHP12,MDO14,Pe73,RR08}).

There are numerous studies of the asymptotic variance in the literature. For reversible Markov processes, the asymptotic variance can be presented by a spectral calculation, which brings a lot of applications (see \cite{DL01,KV86,RR97} etc.).
The comparisons 
on efficiency of reversible Markov processes, in terms of the asymptotic variance, has been extensively researched(see e.g. \cite{AL19,HM21+,LM08,Pe73,Ti98}). Recently, there are also some comparison results between reversible and non-reversible Markov processes, see e.g.  \cite{Bi16,CH13,Hw05,SGS10} for discrete-time Markov chains, and \cite{DLP16,HNW15,RS15} for diffusions.
However,  the study of the asymptotic variance of non-reversible Markov processes is still a challenge since  the lack of spectral theory of non-symmetric operators.  Very recently, \cite{HM21+} gives some variational formulas for the asymptotic variance of general discrete-time Markov chains by  solving Poisson equation, and obtains some estimates and comparison results of  the asymptotic variance.

In this paper we extend the results in \cite{HM21+} to the general Markov process by constructing the weak solution of Poisson equation with the help of the semi-Dirichlet form.

Let $X=\{X_t\}_{t\geq0}$ be a positive recurrent (or ergodic) Markov process on a Polish space $(S,\mathcal{S})$, with strongly continuous contraction transition semigroup $\{P_t\}_{t\geq0}$ and stationary distribution $\pi$. Denote $L^2(\pi)$ by the space of square integrable functions with scalar product
$
(u,v):=\int_S u(x)v(x)\pi(dx)
$
and norm $||u||=( u,u)^{1/2}$. Let $L^2_0(\pi)$ be the subspace of functions in $L^2(\pi)$ with mean-zero, i.e.
$$
L^2_0(\pi)=\{u\in L^2(\pi):\ \pi(u):=\int_S u(x)\pi(dx)=0\}.
$$
Denote $(L,\mathscr{D}(L))$ by the infinitesimal generator induced by $\{P_t\}_{t\geq0}$ in $L^2(\pi)$.

 Define a bilinear form associated with $L$ as
$$
\e(u,v)=( -Lu,v),\quad u,v\in\scr{D}(L),
$$
and for  $\alpha\geq0$,
$$
\e_\alpha(u,v):=\e(u,v)+\alpha(u,v),\quad u,v\in\scr{D}(L).
$$
Since $\pi$ is the stationary distribution and $P_{t}$ is $L^{2}$-contractive, for any  $u \in \mathscr{D}(L) $,
$$
\mathscr{E}(u, u)=-\pi(u L u)=\lim_{t\rar0}\frac{\pi(u^2)-\pi\left(\left(P_{t} u\right)^{2}\right)}{t}\ge0,
$$
that is, $\scr{E}$ is non-negative define.

  $(L,\mathscr{D}(L))$ is said to satisfy the sector condition if there exists a constant $K>0$ such that
\begin{equation}\label{weak-sect}
	|\e(u,v)|\leq K\e(u,u)^{1/2}\e(v,v)^{1/2},\quad \text{for  }u,v\in\scr{D}(L).
\end{equation}
Remarkably, if process $X$ is reversible:
	$$
	\pi(\d x)P_t(x,\d y)=\pi(\d y)P_t(y,\d x),\quad \text{for all }t\geq0,\ \pi\text{-a.s. }x,y\in S,
	$$
	then the sector condition is always true with $K=1$ by Cauchy-Schwartz inequality.

Under the sector condition, we can obtain a unique semi-Dirichlet form $(\e,\scr{F})$, where $\scr{F}$ is the completion of $\scr{D}(L)$ with respect to $\bar{\e}_1^{1/2}$ ($\bar{\e}_1$ is the symmetric part of $\e_1$), see \cite[Chapter 1, Theorem 2.15]{MR92}) for more details.



We say that  the semigroup $\{P_t\}_{t\geq0}$ is $L^2$-exponentially ergodic, if there exist constants $C, \lambda_1>0$ such that for  $u\in L_0^2(\pi)$,
$$
\|P_tu\|\leqslant C\|u\|\rm{e}^{-\lambda_1 t}.
$$
It is well known  that when process $X$ is reversible,
$C$ can be chosen as $1$ and (the optimal) $\lambda_1$ is nothing but  the spectral gap:
\begin{equation}\lb{spec-gap}
\lambda_1=\inf\{\e(u,u):\ u\in\scr{F},\pi(u)=0\ \text{and }\pi(u^2)=1\}.
\end{equation}

Now for $f\in L^2_0(\pi)$, we consider the following asymptotic variance for $X$ and $f$:
\be\lb{av-f}
\sigma^2(X,f)=\limsup_{t\rightarrow\infty}\bE_\pi\Big[\Big(\frac{1}{\sqrt{t}}\int_0^tf(X_s)\d s\Big)^2\Big].
\de

Under the $L^2$-exponential ergodicity and the sector condition, our first main result presents a variational formula for the asymptotic variance as follows.

\begin{thm}\lb{main1}
Suppose that the semigroup $\{P_t\}_{t\geq0}$ associated with process $X$ is
$L^2$-exponentially ergodic. Then the limit in \eqref{av-f} exists and is finite for  $f\in L^2_0(\pi)$. If in addition the associated semi-Dirichlet form $(\e,\scr{F})$ satisfies the sector condition, then for $f\in L^2_0(\pi)$,
\begin{equation}\lb{vf-av}
2/\sigma^2(X,f)=\inf_{u\in\scr{M}_{f,1}}\sup_{v\in\scr{M}_{f,0}}\e(u+v,u-v),
\end{equation}
where $\scr{M}_{f,\delta}=\{u\in\scr{F}:(u,f)=\delta\},\ \delta=0,1$.

Particularly, if process $X$ is reversible, then \eqref{vf-av} is reduced to
\begin{equation}\lb{vf-av-r}
2/\sigma^2(X,f)=\inf_{u\in\scr{M}_{f,1}}\e(u,u).
\end{equation}
\end{thm}

\begin{rem}
\begin{itemize}
\item[(1)] For fixed $f\in L^2_0(\pi)$, from the proof below we will see that functions $Gf:=\int_0^\infty P_tf\d t$ and $G^*f:=\int_0^\infty P^*_t f\d t$ are both in $\scr{F}$, here $P^*_t$ is the dual operator of $P_t$ in $L^2(\pi)$. This is a main reason that we need the semi-Dirichlet form $(\e,\scr{F})$ in \eqref{vf-av}. However, if the generator $L$ is bounded in $L^2(\pi)$, then
 $\scr{D}(L)=L^2(\pi)$, so that $Gf,G^*f\in\scr{D}(L)$. In this case,
$$
2/\sigma^2(X,f)=\inf_{\substack{u\in L^2(\pi),\\ \pi(uf)=1}}\sup_{\substack{v\in L^2(\pi),\\ \pi(vf)=0}}((-L)(u+v),u-v).
$$
The proof of this result can be obtained immediately by replacing $P-I$ by $L$ in the proof of \cite[Theorem 1.1]{HM21+}.

\item[(2)] The assumption of the $L^2$-exponential ergodicity of $\{P_t\}_{t\geq0}$ is not too strong for non-reversible Markov processes, since \cite{Ha05} gives a geometrically ergodic Markov chain such that the asymptotic variance is infinite for some $f\in L_0^2(\pi)$.

\item[(3)] Variational formula for the asymptotic variance has been studied in \cite[Chapter 4]{KLO12}. It is based on a variational formula for positive definite operators in analysis and resolvent equations. Here we obtain a new  variational formula.

\end{itemize}
\end{rem}

As a direct application of Theorem \ref{main1}, bound of the mean exit time of the process is obtained. For that, let $\Omega\subset S$ be an open set, denote by $\tau_\Omega=\inf\{t\geq0:X_t\notin\Omega\}$  the first exit time from $\Omega$ of process $X$.

\begin{cor}\lb{exit-time}
Suppose that process $X$ is reversible with $L^2$-exponentially ergodic semigroup $\{P_t\}_{t\geq0}$ and stationary distribution $\pi$. Let $\Omega\subset S$ be an open set with $\pi(\Omega)\in(0,1)$, then
$$
\bE_\pi\tau_\Omega\leq \frac{\pi(\Omega)}{2\lambda_1\pi(\Omega^c)},
$$
 where $\lambda_1$ is the spectral gap defined in $\eqref{spec-gap}$.
\end{cor}

Note that in \cite[Remark 3.6(1)]{HKMW20}, we gave another upper bound for the mean exit time. Explicitly, $\bE_\pi\tau_\Omega\leq 1/(\lambda_1\pi(\Omega^c))$ for open set $\Omega\subset S$ satisfying $\pi(\Omega^c)>0$. It is obvious that the upper bound in Corollary \ref{exit-time} is more precise than that.

For the reversible case,  similar to \cite[Theorem 1.3]{HM21+}, we could derive variational formula \eqref{vf-av-r} without the assumption of the $L^2$-exponential ergodicity. Since the proof is quite similar, we omit it in this paper.

\begin{thm}\lb{main-r}
Suppose that $X$ is a reversible ergodic Markov process with stationary distribution $\pi$. Then for fixed $f\in L^2_0(\pi)$,
$$
2/\sigma^2(X,f)=\inf_{u\in\scr{M}_{f,1}}\e(u,u).
$$
\end{thm}

Note that in Theorem \ref{main-r}, maybe $\sigma^2(X,f)=\infty$ for some $f\in L^2_0(\pi)$.

The remaining part of this paper is organized as follows. In Section \ref{sect-eg} we apply our main result in two situations.
The first application is extending the comparison result for the asymptotic variance of one dimensional diffusions in \cite[Theorem 1]{RR14} to multi-dimensional reversible diffusions.
We note that \cite[Theorem 1]{RR14} is proved by discrete approximation which is different from our idea, and the less assumptions are requested in our proof. Another application is a comparison result between reversible and non-reversible diffusions on Riemannian Manifolds, which shows the asymptotic variance of a non-reversible diffusion is smaller.
The similar result can be found in \cite{DLP16,HNW15}(for example, \cite{HNW15} proves a similar result on compact manifolds by using a spectral theorem), we provide a complete different proof by the new variational formula.
 Finally, the proofs of Theorem \ref{main1} and Corollary \ref{exit-time} are given in Section 3.



\section{Applications}\lb{sect-eg}
\subsection{Reversible  diffusions}\lb{Lang-rev}


 First, we recall the comparison theorem proved in \cite[Theorem 1]{RR14}.
Fix a $C^1$  probability density function $\mu:[I_1,I_2]\rightarrow (0,\infty)$,  where $-\infty\leq  I_1<I_2\leq\infty$.
 	Given a $C^1$ positive function $\eta$ on $[I_1,I_2]$ and consider a one-dimensional Langevin diffusion:
 	$$\d X^\eta_{t}=\eta\left(X^\eta_t\right) \d B_{t}+\left(\frac{1}{2} \eta^{2}\left(X^\eta_{t}\right) \log \mu^{\prime}\left(X^\eta_{t}\right)+\eta\left(X^\eta_{t}\right) \eta^{\prime}\left(X^\eta_{t}\right)\right) \d t.$$
 	Under some additional conditions (see \cite[Page 133]{RR14}), \cite{RR14} proves that for any $f\in L^2_0(\mu)$, and two $C^1$ positive functions $\eta,\eta_1$  on $[I_1,I_2]$ such that   $\eta_1(x)\leq \eta(x)$ for all $x\in[I_1,I_2]$, 	
 $$
 		\sigma^2(X^{\eta_1},f)\geq \sigma^2(X^{\eta},f).
$$
Note that in \cite{RR14}, the above conclusion is proved by discrete approximation. In fact, we can obtain the above result  by a direct calculation as follows. For convenience, we only consider the case on half-line.

Fix a $C^1$ probability density function ${\pi}:[0,\infty)\rightarrow (0,\infty)$.
Given a $C^1$ positive function $a$ on $[0,\infty)$ and consider a one-dimensional diffusion  $X^a$ with reflecting boundary 0 and generator:
\begin{equation}\label{1d-generator}
  L_a=a(x)\frac{\mathrm{d}^2}{\mathrm{d}x^2}+b(x)\frac{\mathrm{d}}{\mathrm{d}x},
\end{equation} 
where $b(x)=a(x)({\pi}'(x)/{\pi}(x))+a'(x)$. Let $\pi(\d x)=\pi(x)\d x$. It is easy to see that $L_a$ is symmetric on $L^2({\pi})$. Choose a point $x_0>0$ and set $$c(x)=\int_{x_0}^{x}\frac{b(y)}{a(y)}\mathrm{d}y\quad\text{and}\quad \varphi(x)=\int_{0}^x\mathrm{e}^{-c(y)}\d y.$$
So we have
\begin{equation}\label{pi}
	{\pi}(x)= {{\mathrm e}^{c(x)}{\pi}(x_0)a(x_0)}/{a(x)}.
\end{equation}
Assume that  $X^a$ is non-explosive,
that is,
$$\int_{0}^{\infty}\varphi'(y)\pi([0,y])\mathrm{d}y=\infty,
$$
then  $X$ is ergodic with stationary distribution $\pi(\d x)$(see e.g. \cite[Table 5.1]{cmf05}).

For fixed function $f\in L^2_0(\pi)$, consider Poisson equation $-L_au=f$.  By some direct calculations and \eqref{pi}, the equation has strong solution
$$			u(x)=\int_{0}^{x}\mathrm{e}^{-c(y)}\left(\int_{y}^{\infty}f(z)\frac{\mathrm{e}^{c(z)}}{a(z)}\mathrm{d}z\right)\mathrm{d}y
=\frac{1}{\pi(x_0)a(x_0)}\	\int_{0}^{\infty}f(z)\varphi(x\wedge z)\pi(\d z).
$$

Since $\sigma^2(X,f)=2(u,f)$ by Lemma \ref{Gf-u} and \eqref{representation} below, from $\pi(f)=0$ and the integration by parts we have that 
\be\lb{exp-reps}
\aligned \frac{1}{2}\sigma^2(X^a,f)&=\frac{1}{ a(x_0)\pi(x_0)}\int_{0}^{\infty}\int_{0}^{\infty}f(x)f(y)\varphi(x\wedge y)\pi(\d y)\pi(\d x)\\ &=\frac{2}{ a(x_0)\pi(x_0)}\int_{0}^{\infty}\varphi(x)f(x)\int_{x}^{\infty}f(y)\pi(\d y)\pi(\d x)\\
&=-\frac{2}{a(x_0)\pi(x_0)}\int_{0}^{\infty}\varphi(x)f(x)\int_{0}^{x}f(y)\pi(\d y)\pi(\d x)\\
&=-\frac{1}{ a(x_0)\pi(x_0)}\int_{0}^{\infty}\varphi(x)\Big[\Big(\int_0^x f(y)\pi(\d y)\Big)^2\Big]'\d x\\
&=\int_0^\infty\Big(\int_0^x f(y)\pi(\d y)\Big)^2\frac{1}{a(x)\pi(x)}\d x.
\endaligned
\de

Using the above representation, we obtain the following comparison theorem directly.

\begin{thm}\lb{compa-1}
Let $a,a_1$ be two $C^1$ positive function on $[0,\infty)$. Then Langevin diffusions $X^a$ and $X^{a_1}$, with generators of form \eqref{1d-generator}, possess the same stationary distribution $\pi$.
Moreover, if $a\geq a_1$, then for any $f\in L_0^2(\pi)$,
$$
		 \sigma^2(X^{a},f)\leq \sigma^2(X^{a_1},f).
$$
In particular, for fixed $f\in L^2_0(\pi)$, $\sigma^2(X^{ka},f)$ is non-increasing for $k\in(0,\infty)$.	
\end{thm}

For multi-dimensional reversible diffusion processes, {explicit representation \eqref{exp-reps} for the asymptotic variance is difficult to obtain. However, we could use Theorem \ref{main-r} to get the similar comparison result as follows.


Let $V\in C^2(\bR^d)$ with $\int_{\bR^d} \mathrm{e}^{V(x)}\mathrm{d} x<\infty$. Consider the reversible diffusion process $X^{A}$ generated by  elliptic operator
$$L_{A}=\sum_{i, j} a_{i j}(x) \frac{\partial^{2}}{\partial x_{i} \partial x_{j}}+\sum_{i} b_{i}(x) \frac{\partial}{\partial x_{i}},$$
where $A(x)=\left(a_{i j}(x)\right)_{1\leq i,j\leq d},x\in\bR^d$ are positive definite matrices with $a_{i j}\in  C^{2}\left(\mathbb{R}^{d}\right)$ and
$$
b_{i}(x)=\sum_{j} a_{i j}(x) \frac{\partial}{\partial x_{j}} V(x)+\sum_{j} \frac{\partial}{\partial x_{j}} a_{i j}(x).
$$
Assume that $X^A$ is non-explosive. By \cite[Theorem 4.2.1]{DZ96}, we see that process $X^{A}$ is ergodic with stationary distribution
$$\pi(\d x):=\frac{\mathrm{e}^{V(x)}}{\int_{\bR^d} \mathrm{e}^{V(y)}\mathrm{d} y}\mathrm{d} x.$$

Denote by $(\e_{A}(\cdot,\cdot),\scr{F}_A)$  the Dirichlet form associated with the process $X^{A}$. Explicitly, we see that
\begin{equation}\lb{df-r}
\e_{A}(u,v)=\int_{\bR^d}\nabla u\cdot A\nabla v\d \pi,\quad \text{for}\ u,v\in \scr{F}_A:=\{u\in L^2(\pi): \e_A(u,u)<\infty\}.
\end{equation}

\begin{thm}\lb{compa-r}
Let $V\in C^2(\bR^d)$ with $\int_{\bR^d} \mathrm{e}^{V(x)}\mathrm{d} x<\infty$, $A(x)=(a_{ij}(x))_{1\leq i,j\leq d}$ and $A_1(x)=(a^1_{ij}(x))_{1\leq i,j\leq d}$, $x\in \bR^d$ be positive definite matrices satisfying $a_{ij}, a^1_{ij}\in C^2(\bR^d)$ for $1\leq i,j\leq d$. Suppose that $A_1\leq A$ in the sense that $A(x)-A_1(x)$ is non-negative definite for all $x\in\bR^d$.
Then for any $f\in L^2_0(\pi)$,
\begin{equation}\lb{A12}
\sigma^2(X^{A_1},f)\geq \sigma^2(X^{A},f).
\end{equation}
 In particular, for fixed $f\in L^2_0(\pi)$, $\sigma^2(X^{kA},f)$ is non-increasing for $k\in(0,\infty)$.
\end{thm}
\begin{proof}
Since $A_1\leq A$, by \eqref{df-r} it is easy to check that $\scr{F}_{A_1}\supseteq \scr{F}_{A}$ and
$$
\e_{A_1}(u,u)\leq \e_{A}(u,u)\quad \text{for all }u\in \scr{F}_{A}.
$$

Fix $f\in L^2_0(\pi)$. The inequality \eqref{A12} is trivial when $\sigma^2(X^{A_1},f)=\infty$. Now assume that $\sigma^2(X^{A_1},f)<\infty$. It follows from Theorem \ref{main-r} that
$$
2/\sigma^2(X^{A_1},f)=\inf_{\substack{u\in\scr{F}_{A_1},\\\pi(fu)=1}}\e_{A_1}(u,u)
\leq \inf_{\substack{u\in\scr{F}_{A},\\\pi(fu)=1}}\e_{A}(u,u)=2/\sigma^2(X^{A},f).
$$
Hence, the proof is completed.
\end{proof}


\subsection{Non-reversible diffusions on Riemannian Manifolds}\lb{nonrev-diff}
In this section,we turn to non-reversible case. Let $M$ be a connected, complete Riemannian manifold  with empty boundary or convex boundary, and $\lan\cdot,\cdot\ran$ be the inner product  under the Riemannian metric.
Denote  $\d x$ and $\Delta$ by the Riemannian volume and  Laplace operator on $M$, respectively.

Let $\pi(\d x):=e^{-U(x)} \d x$ be a probability measure on $M$ with potential function $U\in C^{2}(M)$.
We consider the following diffusion operator:
\be\label{Diff}
\mathfrak{L} \varphi=\Delta \varphi-\lan\nabla U-Z,\nabla \varphi\ran,
\de
where $Z$ is a $C^1$ vector field on $M$.
Denote by $\mathfrak{L}^{*}$  the dual operator of $\mathfrak{L}$ on  $L^{2}(\pi)$:
$$\mathfrak{L}^{*}\varphi=\Delta\varphi-\lan\nabla U+Z,\nabla\varphi\ran-(\operatorname{div} Z-\lan\nabla U, Z\ran)\varphi,$$
where $\operatorname{div}$ is the divergence operator.
It is well known that $\pi$ is the invariant measure of $\mathfrak{L}$ if and only if $(\mathfrak{L}^{*}1,\varphi)=0$  for $\varphi\in C_{0}^{\infty}(M),$ i.e.,
$$\int_M(\operatorname{div} Z-\lan\nabla U, Z\ran)\varphi\d \pi=\int_M\operatorname{div}( Z\mathrm{e}^{-U}) \varphi\d x=0.$$
From now on we assume that
\begin{equation}\label{A0}
\operatorname{div}( Z\mathrm{e}^{-U})\equiv 0.
\end{equation}
Then by \cite[Corollary 3.6]{brw01}, the diffusion $X$ with generator $\mathfrak{L}$ is ergdoic with stationary distribution $\pi$.

Denote the symmetric part of  $\mathfrak{L}$ with respect to $\pi$ by  $\bar{\mathfrak{L}}:=\Delta-\lan\nabla U,\nabla\ran.$
and let $\bar{X}$ be the diffusion generated by $\bar{\mathfrak{L}}$.

 Define $(\e,\scr{F})$ as the semi-Dirichlet form generated by $\mathfrak{L}$, and denote its symmetric part and antisymmetric part by $\bar{\e},\ \widehat{\e}$ respectively. So from the integration by parts and \eqref{A0}, we have
$$
\bar{\e}(\varphi, \phi)= \int_M\lan \nabla \varphi,\nabla \phi\ran \d \pi\quad \text{and}\quad \widehat{\e}(\varphi, \phi)=\int_M \phi\lan Z,\nabla\varphi\ran \d\pi  \quad \varphi, \phi \in C_{0}^{\infty}(M).
$$
 Indeed, it is easy to check that $(\bar{\e},\scr{F})$ is the Dirichlet form generated by $\bar{\mathfrak{L}}$.

We suppose that the following {\bf Assumption A} holds:
\begin{itemize}

\item[(\rm A1)] $|\Delta U| \leq \epsilon_*|\nabla U|^{2}+C_U$ for some $\epsilon_*<1$  and $C_U \geq 0;$

\item[(\rm A2)] there is a constant $K$ such that $|Z| \leq K(|\nabla U|+1)$;

\item[(\rm A3)] the symmetric Dirichlet form $(\bar{\e},\scr{F})$ satisfies the Poincar\'{e} inequality, i.e., there exists a constant $\lambda_1>0$ such that
$$\|\varphi\|^2\leq \lambda_1^{-1}\bar{\e}(\varphi,\varphi)\quad \text{for all } \varphi\in\scr{F},$$
where $\|\cdot\|$ is $L^2(\pi)$-norm.
\end{itemize}

We note that (A3) is equivalent to the $L^2$-exponential ergodicity of semigroup of diffusion $\bar{X}$.

\begin{lem}\lb{Lang-sec}
If {\bf Assumption A} and \eqref{A0} hold, then $(\mathscr{E},\scr{F})$ satisfies the sector condition \eqref{weak-sect}. Therefore, Theorem \ref{main1} holds for the diffusion $X$.
\end{lem}
\begin{proof}

Since $(\bar{\e},\scr{F})$ is symmetric, it  satisfies the sector condition, we only need to check the sector condition for the antisymmetric part $\widehat{\e}$.

Fix $\phi,\varphi\in C^\infty_0(M)$. By	Cauchy-Schwarz inequality and (A2) we have
\begin{equation}\lb{c-s}
	\begin{split}
		\int_M\lan\phi Z,\nabla \varphi\ran \d \pi&\leq K\int_M (|\nabla U|+1)|\phi\nabla \varphi|\d \pi\\
		&\leq K\int_M |\phi\nabla \varphi|d \pi+K\bar{\e}(\varphi,\varphi)^{1/2}	\||\nabla U| \phi\|	.
	\end{split}
\end{equation}
For the last term above, the integration by parts on manifold, Cauchy-Schwarz inequality and (A1)  yield that
\begin{equation}\label{i-p and c-s}
\begin{aligned}
	\||\nabla U| \phi\|^2 &=-\int_M \lan\phi^{2} \nabla U , \nabla e^{-U}\ran \d x
	=\int_M \operatorname{div}\left(\phi^{2} \nabla U\right) e^{-U} \d x \\
	&=\int_M \lan2 \phi \nabla \phi , \nabla U\ran \d \pi+\int_M \Delta U \phi^{2} \d \pi\\
&=\int_M \lan2 \phi \nabla \phi , \nabla U\ran \d \pi+\int_M\left(\epsilon_* |\nabla U|^{2}+C_U\right) \phi^{2} \d \pi.
\end{aligned}
\end{equation}

Now fix $\varepsilon>0$ such that $\epsilon_*+\varepsilon<1$. Combining inequality $|x y| \leq\left(x^{2} / \varepsilon+\varepsilon y^{2}\right) / 2$ with \eqref{i-p and c-s} and (A3) we have
$$
\begin{aligned}
		\||\nabla U| \phi\|^2 &\leq 2 \int |\nabla \phi||\phi\nabla U| \mathrm{d} \pi+\int\left(\epsilon_* |\nabla U|^{2}+C_U\right) \phi^{2} \mathrm{~d} \pi \\
	&\leq \frac{1}{\varepsilon}\bar{\e}(\phi,\phi)+(\epsilon_*+\varepsilon)\|\phi|\nabla U|\|^{2}+C_U\|\phi\|^{2}\\
	&\leq \frac{1}{\varepsilon}\bar{\e}(\phi,\phi)+(\epsilon_*+\varepsilon)\|\phi|\nabla U|\|^2+C_U\lambda_1^{-1}\bar{\e}(\phi,\phi),
\end{aligned}
$$
which implies that
$$
\||\nabla U| \phi\|^2\leq\frac{\lambda_1+C_U\varepsilon}{(1-\epsilon_*-\varepsilon)\varepsilon\lambda_1}\bar{\e}(\phi,\phi).
$$
Combining this with \eqref{c-s} and (A3), we obtain that $\widehat{\e}$ satisfies the sector condition on $\scr{F}$.
\end{proof}

From Lemma \ref{Lang-sec} and Theorem \ref{main1}, we obtain the following comparison result.

\begin{thm}\lb{compr-diff}
Suppose that {\bf Assumption A} holds. Then for any $f\in L^2_0(\pi)$,
$$
\sigma^2(X,f)\leq \sigma^2(\bar{X},f).
$$
\end{thm}
\begin{proof}
Since the conditions in Theorem \ref{main1} are satisfied by Lemma \ref{Lang-sec}, 
we obtain by taking $v=0$ that
$$
\aligned
2/\sigma^2(X,f)&=\inf_{u\in\scr{M}_{f,1}}\sup_{v\in\scr{M}_{f,0}}\e(u+v,u-v)\\
&\geq \inf_{u\in\scr{M}_{f,1}}\e(u,u)=\inf_{u\in\scr{M}_{f,1}}\bar{\e}(u,u)=2/\sigma^2(\bar{X},f).
\endaligned
$$
\end{proof}

\begin{rem}
Similar comparison result in Theorem \ref{compr-diff} can be found in \cite{DLP16,HNW15}. For example, \cite{HNW15} proves the comparison theorem by using a spectral theorem (see \cite[Section 3.4.3]{HNW15}). Here we provide a completely different proof by the new variational formula.
\end{rem}

\begin{exm}(\cite[Example 5.2]{KS14})
	Let $M=\mathbb{R}^{2}$, potential function $U(x)=(1 / 2 \pi) e^{-|x|^{2} / 2}$ and vector field
	$$
	Z=-c x_{2} \frac{\partial}{\partial x_{1}}+c x_{1} \frac{\partial}{\partial x_{2}},
	$$ where $c$ is a positive constant. Consider the 2-dimensional Ornstein-Uhlenbeck diffusion with rotation:
	$$
	\mathfrak{L}:=\frac{1}{2}\left(\frac{\partial^{2}}{\partial x_{1}^{2}}+\frac{\partial^{2}}{\partial x_{2}^{2}}\right)-(x_{1}+c x_{2}) \frac{\partial}{\partial x_{1}}-(x_{2}-cx_1) \frac{\partial}{\partial x_{2}}.
	$$
	 Its invariant probability measure is $\pi(d x)=(1 / 2 \pi) \mathrm{e}^{-|x|^{2} / 2} d x$. The symmetric part of $\mathfrak{L}$ with respect to $\pi$ is
	$$\bar{\mathfrak{L}}:=\left(\frac{\partial^{2}}{\partial x_{1}^{2}}+\frac{\partial^{2}}{\partial x_{2}^{2}}\right)-x_{1} \frac{\partial}{\partial x_{1}}-x_{2} \frac{\partial}{\partial x_{2}}. $$

	Since the symmetric Ornstein-Uhlenbeck diffusion generated by $\bar{\mathfrak{L}}$ is exponentially ergodic, (A3) is satisfied.  A direct calculation shows that $ \operatorname{div}( Z\mathrm{e}^{-U})=0$ and (A1), (A2) are satisfied. Hence,  Theorems \ref{main1} and \ref{compr-diff} are valid.  
\end{exm}


\section{Proofs of Theorem \ref{main1} and Corollary \ref{exit-time}}\lb{proofs}

Recall that $X=\{X_t\}_{t\geq0}$ is a positive recurrent (or ergodic) Markov process on a Polish space $(S,\mathcal{S})$, with strongly continuous contraction transition semigroup $\{P_t\}_{t\geq0}$ and stationary distribution $\pi$.
$(L,\mathscr{D}(L))$, $(\e,\scr{F})$ are its associated infinitesimal generator in $L^2(\pi)$ and semi-Dirichlet form, respectively.
For fixed $f\in L^2_0(\pi)$, we want to study the 
 asymptotic variance of $X$ and $f$ defined in \eqref{av-f}. Indeed, from \cite[Section 2.5]{KLO12}, we see that the asymptotic variance can be represented by $P_t$ as follows:
\begin{equation}\lb{av-fp}
\sigma^2(X,f)=2\lim_{t\rightarrow\infty}\int_0^t (1-\frac{s}{t})( P_s f,f)\d s.
\end{equation}

To prove Theorem \ref{main1}, first we do some preparations.
For any $\alpha>0$, set $G_\alpha f=\int_0^\infty \text{e}^{-\alpha s}P_sf\d s$ for $f\in L^2(\pi)$. From \cite[Chapter 1, Proposition 1.10]{MR92} we see that $(G_\alpha)_{\alpha>0}$ is the strong continuous contraction resolvent associated to $L$ and $G_\alpha f\in\scr{D}(L)$ for all  $f\in L^2(\pi)$.
If the semigroup $\{P_t\}_{t\geq0}$ is $L^2$-exponentially ergodic, then it is known that $G f:=\int_0^\infty P_sf\d s\in L^2(\pi)$ for $f\in L_0^2(\pi)$.

\begin{lem}\lb{Gf-u}
Suppose that the semigroup $\{P_t\}_{t\geq0}$ is
$L^2$-exponentially ergodic and its corresponding semi-Dirichlet form $(\e,\scr{F})$ satisfies the sector condition \eqref{weak-sect}. Then for all $f\in L^2_0(\pi)$, we have $Gf\in\scr{D}(L)$ and
$$
\e(Gf,u)=(f,u),\quad  u\in \scr{F}.
$$

\end{lem}
\begin{proof}
We first prove that  $Gf\in \scr{D}(L)$ for all $f\in L^2_0(\pi)$.
Note that the generator $L$ is closed and densely defined, that is, $\scr{D}(L)$ is complete with respect to the graph norm $\|Lu\|+\|u\|,u\in\scr{D}(L)$ (see e.g. \cite[Chapter 1, Proposition 1.10]{MR92}). Thus for fixed $f\in L^2_0(\pi)$, we only need to prove that $\|G_{1/n}f-Gf\|\rightarrow 0$ as $n\rightarrow \infty$ and $\{G_{1/n}f\}_{n\geq 1}$ is a Cauchy sequence under $\|L\cdot\|$ by $G_{1/n}f\in \scr{D}(L),n\geq1$. Indeed, it follows from $L^2$-exponential ergodicity and H\"older inequality that
\be\lb{G_1/n}
\aligned
\|G_{1/n} f-Gf\|&= \Big\|\int_0^\infty(1-{\rm e}^{-s/n})P_sf\d s\Big\|
\leq\int_0^\infty (1-{\rm e}^{-s/n})\| P_sf\|\d s\\
&\leq C\|f\|\int_0^\infty(1-{\rm e}^{-s/n}){\rm e}^{-\lambda_1 s}\d s\\
&= C\|f\|\fr{1/n}{\lmd_1(\lmd_1+1/n)}\rightarrow 0,\quad \text{as }n\rightarrow \infty.
\endaligned
\de
On the other hand, since $Lf=(\alpha-G_\alpha^{-1})f$ for all $\alpha>0$ and $f\in L^2_0(\pi)$, we have
$$
\aligned
\|L(G_{1/n}f-G_{1/m}f)\|&=\|(1/n-G_{1/n}^{-1})G_{1/n}f-(1/m-G_{1/m}^{-1})G_{1/m}f\|\\
&=\|\frac{1}{n}G_{1/n}f-\frac{1}{m}G_{1/m}f\|\\
&\leq \frac{1}{n}\|G_{1/n}f-G_{1/m}f\|+|\frac{1}{n}-\frac{1}{m}|\|G_{1/m}f\|\\
&\rightarrow 0,\quad \text{as }n,m\rightarrow\infty.
\endaligned
$$
Therefore $Gf\in \scr{D}(L)$.

Next we prove that $-LGf=f$ for $f\in L^2_0(\pi)$. Arguing similarly as we did in \eqref{G_1/n}, $\lim_{\bt\rar0}\|G_\bt f-Gf\|=0$ for  $f\in L^2_0(\pi)$. Combining this fact with the property $G_\beta-G_\alpha=(\alp-\bt)G_\alpha G_\beta$, we obtain that
$$
Gf-G_\alpha f=\alpha G_\alpha Gf,\quad \text{for }\alpha>0 ,\ f\in L^2_0(\pi).
$$
Using this equality and the fact $Gf\in \scr{D}(L)$ shows that for any $\alpha>0$ and $ f\in L^2_0(\pi)$,
$$
G_\alpha(-LGf)=G_\alpha(G_\alpha^{-1}-\alpha)Gf=Gf-\alpha G_\alpha Gf=G_\alpha f.
$$
That is, $-LGf=f$ for all $f\in L_0^2(\pi)$.

From above analysis and \cite[Chapter 1, Corollary 2.10]{MR92} we could obtain that for any $f\in L^2_0(\pi),u\in\scr{F}$,
$$
\e(Gf,u)=((-L)Gf,u)=( f,u).
$$
\end{proof}

We now proceed to prove Theorem \ref{main1}.

\medskip
\noindent{\bf Proof of Theorem \ref{main1}.}
For fixed $f\in L^2_0(\pi)$, we first claim that the limit in \eqref{av-f}, i.e. \eqref{av-fp}, exists and $\sigma^2(X,f)=2(Gf,f)<\infty$. Indeed, for $t>0$,
$$
2\int_0^t (1-\frac{s}{t})( P_s f,f)\d s=2\int_0^t( P_sf,f)\d s-\frac{2}{t}\int_0^ts( P_sf,f)\d s.
$$
Since $\{P_t\}_{t\geq 0}$ is $L^2$-exponentially ergodic, we arrive at

	$$
		\aligned
			\frac{1}{t}\Big|\int_0^ts( P_sf,f)\d s\Big|&\leq \frac{1}{t} \int_0^t s\|P_s f\|\|f\|\d s\leq \frac{C\|f\|^2}{t} \int_0^t s\mathrm{e}^{-\lambda_1 s}\d s\\
			&\leq\frac{1-(1+\lambda_1 t)\mathrm{e}^{-\lambda_1 t}}{t}C\|f\|^2  \rightarrow0, \ \text{as}\ t\rightarrow\infty,
		\endaligned
	$$
and
$$
\Big|\int_t^\infty ( P_sf,f)\d s\Big|\leq C\|f\|^2\int_t^\infty \text{e}^{-\lambda_1 s}\d s\rightarrow 0,\quad \text{as }t\rightarrow\infty.
$$
Therefore, by combining above analysis, we obtain that the limit in \eqref{av-fp} exists and
\begin{equation*}\label{representation}
	\sigma^2(X,f)=2\int_0^\infty( P_sf,f)\d s<\infty.
\end{equation*}

By the Fubini-Tonelli’s theorem and $L^2$-exponential ergodicity again we get
$$
\int_0^\infty(P_sf,f)\d s=\int_0^\infty\int_S fP_sf\d\pi\d s=\int_S\int_0^\infty fP_sf\d s\d\pi=( Gf,f).
$$
Thus
\begin{equation}\label{representation}
\sigma^2(X,f)=2(Gf,f)<\infty.
\end{equation}
To prove \eqref{vf-av}, we set $w=Gf/(Gf,f),\ w^*=G^*f/( Gf,f)$ and $u_0=(w+w^*)/2, v_0=(w-w^*)/2$. Then $u_0\in \scr{M}_{f,1}$ and $v_0\in \scr{M}_{f,0}$ by noting
$$(Gf,f)=\int_0^\infty(P_sf,f)\d s=\int_0^\infty(f,P^*_sf)\d s=( G^*f,f).
$$

Now  let $v_1=v-v_0$ for any $v\in\scr{M}_{f,0}$. By the definition of $w,w^*,v_0$ and Lemma \ref{Gf-u}, we have $\pi(v_1f)=0$ and
$$
\e(v_1,w^*)=\e(w,v_1)=\frac{1}{( Gf,f)}\e(Gf,v_1)=\frac{1}{( Gf,f)}( f,v_1)=0.
$$
Therefore, using this fact with $\e(w,w^*)=1/(Gf,f)$ and $\e(u,u)\geq0$ for all $u\in\scr{F}$ gives that
$$
\e(u_0+v,u_0-v)=\e(w-v_1,w^*+v_1)=\e(w,w^*)-\e(v_1,v_1)\leq 1/( Gf,f),
$$
which implies that
\be\lb{geq}
1/( Gf,f)\geq\inf_{u\in\scr{M}_{f,1}}\sup_{v\in\scr{M}_{f,0}}\e(u+v,u-v).
\de

For the converse inequality,  let $u_1=u-u_0$ for any $u\in\scr{M}_{f,1}$. Since $u_0\in\scr{M}_{f,1}$, we also have $\pi(u_1f)=0$. Similar  argument shows that
$$
\e(u+v_0,u-v_0)=\e(w+u_1,w^*+u_1)=\e(w,w^*)+\e(u_1,u_1)\geq 1/( Gf,f).
$$
Therefore,
\be\lb{leq}
1/(Gf,f)\leq\inf_{u\in\scr{M}_{f,1}}\sup_{v\in\scr{M}_{f,0}}\e(u+v,u-v).
\de
So we obtain \eqref{vf-av} by combining \eqref{geq}, \eqref{leq} and the fact $\sigma^2(X,f)=2(Gf,f)$.

When process $X$ is reversible,  $\e(\cdot,\cdot)$ is symmetric, i.e.,
$$
\e(u,v)=\e(v,u),\quad \text{for }u,v\in\scr{F}.
$$
Thus 
$$
\e(u+v,u-v)=\e(u,u)-\e(v,v)\leq \e(u,u).
$$
That is, the supremum in \eqref{vf-av} is attained by $v=0$ for any fixed $u\in\scr{M}_{f,1}$. Hence, we obtain \eqref{vf-av-r}.
\qed

By using Theorem \ref{main1}, we prove Corollary \ref{exit-time} as follows.

\medskip
\noindent{\bf Proof of Corollary \ref{exit-time}.}
Fix an open set $\Omega\subset S$ with $\pi(\Omega)\in (0,1)$. It follows from \cite[Theorem 3.3]{HKMW20} that
\be\label{va-exit}
1/\bE_\pi\tau_\Omega=\inf_{u\in \cN_{\Omega,1}}\e(u,u),
\de
where 	$
\cN_{\Omega,1}:=\{u\in\mathscr{F}:u|_{\Omega^c}=0\ \text{and }\pi(u)=1\}.
$
Take $$f=\frac{\mathbf{1}_\Omega-\pi(\Omega)}{1-\pi(\Omega)}.$$
It is easy to check that
$\pi(f)=0$ and $\|f\|^2=\pi(\Omega)/\pi(\Omega^c).$
Notice that for any $u\in \cN_{\Omega,1},$ by simple calculation we have
$\pi(uf)=1$, thus $u\in \cM_{f,1}$. So we see that $\cN_{\Omega,1}\subset \cM_{f,1}$.
Combining this fact with \eqref{vf-av-r} and \eqref{va-exit}, we obtain that
$$2/\sigma^2(X,f)=\inf_{u\in\scr{M}_{f,1}}\e(u,u)\leq \inf_{u\in \cN_{\Omega,1}}\e(u,u)=1/\bE_\pi\tau_\Omega.$$
That is, $\bE_\pi\tau_\Omega\leq \sigma^2(X,f)/2.$
Moreover, from the reversibility and $L^2$-exponential ergodicity we have
$$
\sigma^2(X,f)/2=\int_0^\infty (P_sf,f)\d s\leq \int_0^\infty \|P_sf\|\|f\|\d s\leq \|f\|^2/\lambda_1.
$$
Hence,
$$
\bE_\pi\tau_\Omega\leq \frac{\|f\|^2}{2\lambda_1}=\frac{\pi(\Omega)}{2\lambda_1\pi(\Omega^c)}.
$$
\qed


{\bf Acknowledgement}\
Lu-Jing Huang acknowledges support from NSFC (No. 11901096), NSF-Fujian(No. 2020J05036), the Program for Probability and Statistics: Theory and
Application (No. IRTL1704), and the Program for Innovative Research Team in Science and Technology in Fujian Province University (IRTSTFJ). Yong-Hua Mao and Tao Wang acknowledge support by the National Key R\&D  Program of China (2020YFA0712900) and the National Natural Science
Foundation of China (Grant No.11771047).


\bibliographystyle{plain}
\bibliography{jump_drift}

\vskip 0.3truein

{\bf Lu-Jing Huang:}
College of Mathematics and Informatics, Fujian Normal University, Fuzhou, 350007, P.R. China. E-mail: \texttt{huanglj@fjnu.edu.cn}

\medskip
{\bf Yong-Hua Mao:}
Laboratory of Mathematics and Complex Systems(Ministry of Education), School of Mathematical Sciences, Beijing Normal University, Beijing 100875, P.R. China. E-mail: \texttt{maoyh@bnu.edu.cn}

\medskip
{\bf Tao Wang:}
Laboratory of Mathematics and Complex Systems(Ministry of Education), School of Mathematical Sciences, Beijing Normal University, Beijing 100875, P.R. China. E-mail: \texttt{wang\_tao@mail.bnu.edu.cn}

\end{document}